\documentclass[12pt,reqno]{amsart}
\usepackage{amsthm}
\usepackage{amssymb}
\usepackage{graphics}
\usepackage{tikz}
\usetikzlibrary{shapes,backgrounds,calc}
\usepackage{latexsym}
\usepackage{multicol}
\usepackage{verbatim,enumerate}
\usepackage{accents}
\usepackage{cite}
\usepackage{multirow}
\usepackage{bigstrut}
\usepackage{amsthm}
\usepackage{amssymb}
\usepackage{graphics}
\usepackage{tikz}
\usetikzlibrary{shapes,backgrounds,calc}
\usepackage{latexsym}
\usepackage{multicol}
\usepackage{verbatim,enumerate}
\usepackage{accents}
\usepackage{cite}
\usepackage{array}
\usepackage[colorlinks=true, linkcolor=blue, citecolor=blue, urlcolor=black]{hyperref}
\usepackage{hyperref}
\usepackage{amsmath, amscd,url}
\usepackage{multirow}
\usepackage{bigstrut}
\usepackage{longtable}

\usepackage{stackengine}

\newcommand{\Q}{\mathbb{Q}}
\newcommand{\Z}{\mathbb{Z}}
\newcommand{\N}{\mathbb{N}}
\usepackage{hyperref}
\usepackage{amsmath, amscd,url}
\usepackage{longtable}

\advance\textwidth by 1.3in \advance\oddsidemargin by -.6in \advance\evensidemargin by -.6in
\theoremstyle{definition}

\newtheorem{theorem}{Theorem}[section]
\newtheorem{lemma}[theorem]{Lemma}
\newtheorem{corollary}[theorem]{Corollary}
\theoremstyle{definition}

\newtheorem{proposition}[theorem]{Proposition}

\theoremstyle{remark}
\newtheorem{remark}[theorem]{Remark}

\theoremstyle{definition}

\newcounter{cnt}
 \makeatletter
\def\mydggeometry{\makeatletter\dg@YGRID=1\dg@XGRID=20\unitlength=0.003pt\makeatother}
\makeatother \theoremstyle{remark}

\numberwithin{equation}{section}
\let\bwdg\bigwedge
\def\bigwedge{{\textstyle\bwdg}}

\newcommand{\nc}{\newcommand}
\newcommand{\rnc}{\renewcommand}

\nc{\cal}{\mathcal} \nc{\goth}{\mathfrak} \rnc{\bold}{\mathbf}

\nc\bomega{{\mbox{\boldmath $\omega$}}} \nc\bpsi{{\mbox{\boldmath $\Psi$}}}
 \nc\balpha{{\mbox{\boldmath $\alpha$}}}
 \nc\bpi{{\mbox{\boldmath $\pi$}}}
 \nc\bvpi{{\mbox{\boldmath $\varpi$}}}
\nc\chara{\operatorname{ch}}

  \nc\bxi{{\mbox{\boldmath $\xi$}}}
\nc\bmu{{\mbox{\boldmath $\mu$}}} \nc\bcN{{\mbox{\boldmath $\cal{N}$}}} \nc\bcm{{\mbox{\boldmath $\cal{M}$}}} \nc\blambda{{\mbox{\boldmath
$\lambda$}}}\nc\bnu{{\mbox{\boldmath $\nu$}}}

\makeatletter
\def\section{\def\@secnumfont{\mdseries}\@startsection{section}{1}%
  \z@{.7\linespacing\@plus\linespacing}{.5\linespacing}%
  {\normalfont\scshape\centering}}
\def\subsection{\def\@secnumfont{\bfseries}\@startsection{subsection}{2}%
  {\parindent}{.5\linespacing\@plus.7\linespacing}{-.5em}%
  {\normalfont\bfseries}}
\makeatother

 \nc{\Hom}{\operatorname{Hom}}
  \nc{\mode}{\operatorname{mod}}
\nc{\End}{\operatorname{End}} \nc{\wh}[1]{\widehat{#1}} \nc{\Ext}{\operatorname{Ext}} \nc{\ch}{\text{ch}} \nc{\ev}{\operatorname{ev}}
\nc{\Ob}{\operatorname{Ob}} \nc{\soc}{\operatorname{soc}} \nc{\rad}{\operatorname{rad}} \nc{\head}{\operatorname{head}}

 \nc{\Cal}{\cal} \nc{\Xp}[1]{X^+(#1)} \nc{\Xm}[1]{X^-(#1)}
\nc\boa{\bold a} \nc\bob{\bold b} \nc\boc{\bold c} \nc\bod{\bold d} \nc\boe{\bold e} \nc\bof{\bold f} \nc\bog{\bold g}
\nc\boh{\bold h} \nc\boi{\bold i} \nc\boj{\bold j} \nc\bok{\bold k} \nc\bol{\bold l} \nc\bom{\bold m} \nc\bon{\mathbb n} \nc\boo{\bold o}
\nc\bop{\bold p} \nc\boq{\bold q} \nc\bor{\bold r} \nc\bos{\bold s} \nc\boT{\bold t} \nc\boF{\bold F} \nc\bou{\bold u} \nc\bov{\bold v}
\nc\bow{\bold w} \nc\boz{\bold z}\nc\ba{\bold A} \nc\bb{\bold B} \nc\bc{\mathbb C} \nc\bd{\bold D} \nc\be{\bold E} \nc\bg{\bold
G} \nc\bh{\bold H} \nc\bi{\bold I} \nc\bj{\bold J} \nc\bk{\bold K} \nc\bl{\bold L} \nc\bm{\bold M} \nc\bn{\mathbb N} \nc\bo{\bold O} \nc\bp{\bold
P} \nc\bq{\bold Q} \nc\br{\bold R} \nc\bs{\bold S} \nc\bt{\bold T} \nc\bu{\bold U} \nc\bv{\bold V} \nc\bw{\bold W} \nc\bz{\mathbb Z} \nc\bx{\bold
x} \nc\KR{\bold{KR}} \nc\rk{\bold{rk}} \nc\het{\text{ht }}

\nc\toa{\tilde a} \nc\tob{\tilde b} \nc\toc{\tilde c} \nc\tod{\tilde d} \nc\toe{\tilde e} \nc\tof{\tilde f} \nc\tog{\tilde g} \nc\toh{\tilde h}
\nc\toi{\tilde i} \nc\toj{\tilde j} \nc\tok{\tilde k} \nc\tol{\tilde l} \nc\tom{\tilde m} \nc\ton{\tilde n} \nc\too{\tilde o} \nc\toq{\tilde q}
\nc\tor{\tilde r} \nc\tos{\tilde s} \nc\toT{\tilde t} \nc\tou{\tilde u} \nc\tov{\tilde v} \nc\tow{\tilde w} \nc\toz{\tilde z} \nc\woi{w_{\omega_i}}

\begin{document}
\setcounter{section}{0}
\setcounter{tocdepth}{1}

%%%%%%%%%%%%%%%%%%%%%%%%%%%%%%%%%%%%%%%%%%%%

\title{On the index of power compositional polynomials }

\author[Sumandeep Kaur]{Sumandeep Kaur}
\author[Surender Kumar]{Surender Kumar}
\author[L\'aszl\'o Remete]{L\'aszl\'o Remete}
\address[Sumandeep Kaur]{Department of Mathematics, Panjab University, Chandigarh}
\address[Surender Kumar]{Department of Mathematics, Indian Institute of Technology (IIT), Bhilai}
\address[L\'aszl\'o Remete\footnote{Corresponding author}]{Institute of Mathematics, University of Debrecen, Hungary and ELKH-DE Equations, Functions, Curves and their Applications
Research Group}
\email[Sumandeep Kaur]{sumandhunay@gmail.com}
\email[Surender Kumar]{surenderk@iitbhilai.ac.in}
\email[L\'aszl\'o Remete]{remete.laszlo@science.unideb.hu}

\thanks{ 
	 The second author is grateful to the University Grants Commision, New Delhi for providing financial support in the form of Junior Research Fellowship through Ref No.1129/(CSIR-NET JUNE 2019). The research of the third named author was supported in part by the E\"otv\"os Lor\'and Research Network (ELKH), the Univerity of Debrecen Scientific Research Bridging Fund (DETKA) and by the 2024-2.1.1-EKÖP program from the source of the National Research, Development and Innovation fund.}


\begin{abstract}
% Let $K=\Q(\theta)$ be an algebraic number field with $\theta$ satisfying a monic irreducible polynomial $f(x)$ of degree $n$ over $\Q.$  The polynomial $f(x)$ is said to be monogenic if   $\{1,\theta,\cdots,\theta^{n-1}\}$ is an integral basis of $K.$ To know whether a monic irreducible polynomial is monogenic or not is one of the important problems in algebraic number theory. In an attempt to   answer  this problem for certain family of polynomials,
The index of a monic irreducible polynomial $f(x)\in \Z[x]$ having a root $\theta$ is the index $[\Z_K:\Z[\theta]]$ where $\Z_K$ is the ring of algebraic integers of the number field $K=\Q(\theta)$.  If $[\Z_K:\Z[\theta]]=1$, then $f(x)$ is monogenic. In this paper, we give necessary and sufficient conditions for a monic irreducible power compositional polynomial $f(x^k)$ belonging to $\Z[x]$, to be monogenic. As an application of our results, for a polynomial $f(x)=x^d+A\cdot h(x)\in\Z[x],$ with $d>1, \deg h(x)<d$ and $|h(0)|=1$, we prove that for each positive integer $k$ with $\rad(k)\mid \rad(A)$, the power compositional polynomial $f(x^k)$ is monogenic if and only if $f(x)$ is monogenic, provided that $f(x^k)$ is irreducible. At the end of the paper, we give infinite families of polynomials as examples.

 %L. Jones \cite{LJ}   conjectured that if $\gcd(n,mB)=1$ and $A$ is a prime number, then  the polynomial  $x^n+A(Bx+1)^m\in\Z[x]$ with $n\ge 3$ and $1\le m\le n-1,$ is monogenic if and only if $n^n+(-1)^{n+m}B^n(n-m)^{n-m}m^mA$ is square-free. In this article, we extend and prove this conjecture to a wider class of irreducible polynomials of the type $f(x)=x^n+A(Bx^k+1)^m\in \Z[x]$, when $n=kd, ~d>m\geq 1$, $B\ne0$, $\gcd(d,mBk)=1$ and $\gcd(k,d^d+(-1)^{d+m}B^dm^m(d-m)^{d-m}A)=1.$  We prove that $f(x)$  is monogenic if and only if both  $A$ and $d^{d}+(-1)^{d+m}B^d(d-m)^{d-m}m^mA$  are  square-free and $p^2\nmid A^p-A$ for all primes $p$ dividing $k.$  In particular, L. Jones' conjecture follows directly when $k=1$ and $A$ prime.
\end{abstract}
\maketitle
\section{Introduction and statements of the results}
 Let $f(x)\in\Z[x]$ be a monic irreducible polynomial with a root $\theta$, and set $K=\Q(\theta)$. If $\Z_K$ denotes the ring of algebraic integers of $K$, then the index of $f(x)$  is the index $[\Z_K:\Z[\theta]]$. It is well known that the index of $f(x)$ is always finite. We say that $f(x)$ is monogenic, if the index of $f(x)$ is equal to $1$. Dedekind  gave a simple criterion (cf.  \cite[Theorem 6.1.4]{HC}, \cite{RD}) stated below to verify when a prime $p$ does not divide $[\Z_K : \mathbb Z[\theta]]$.
 
 \noindent\textbf{Dedekind Criterion} { Let $K=\mathbb{Q}(\theta)$ be an
algebraic number field with $f(x)$ as the minimal polynomial of
the algebraic integer $\theta$ over $\mathbb{Q}.$ Let $p$ be a prime and let
$\overline{f}(x) = \overline{g}_{1}(x)^{e_{1}}\ldots \overline{g}_{t}(x)^{e_{t}}$ be
the factorization of $\overline{f}(x)$ as a product of powers of
distinct irreducible polynomials over $\mathbb{Z}/p\mathbb{Z},$
with each $g_{i}(x)\in \mathbb{Z}[x]$ monic. Let $M(x)$ denote the polynomial
$\frac{1}{p}(f(x)-g_{1}(x)^{e_{1}}\ldots g_{t}(x)^{e_{t}})$ with
coefficients from $\mathbb{Z}.$ Then $p$ does not divide
$[\Z_{K}:\mathbb{Z}[\theta]]$ if and only if for each $i,$ we have
either $e_{i}=1$ or $\bar{g}_{i}(x)$ is coprime to
$\overline{M}(x).$}

Using the above criterion, many authors provided necessary and sufficient conditions for the monogenity of different types of polynomials like binomial, trinomials and certain quadrinomials (cf. \cite{df}, \cite{jks}, \cite{jkss}, \cite{j-kk}, \cite{ljj1}). 
%Clearly, a monic irreducible polynomial is monogenic if there is no prime $p$ which divides its index.  
%In this paper, we investigate the necessary and sufficient conditions under which a monic irreducible power compositional polynomial; i.e., a monic irreducible polynomial of the form $f(x^k)\in\Z[x]$ becomes monogenic. 
The coprimality condition given in Dedekind criterion becomes very difficult to handle in the case of higher degree polynomials. In the present paper, for the polynomials of the form $f(x^k)$, using monogenity of $f(x)$, squarefreeness of $f(0)$, we show that the coprimality condition is dependent only on the polynomial $f(x^p)$, where $p\mid k$ not on $f(x^k)$. Basically, our results can easily generates new infinite monogenic polynomials of higher degree (see Section 4). Recently, L. Jones \cite{ljj1} constructed infinite collections of monic Eisenstein polynomials $f(x)\in\Z[x]$ such that the power compositional polynomials $f(x^{d^n})$ are monogenic for all integers $n\geq 0$ and any integer $d > 1$, where $f(x)$ is Eisenstein with respect to every prime divisor of $d$.    

Our main purpose in this paper is to give simpler conditions for the monogenity of an irrreducible power compositional polynomial $f(x^k)$.
This inquiry leads us to Theorem \ref{sufnec} that provides a characterization of such polynomials.  Theorem \ref{sufnec} ensure that $f(x^k)$ is monogenic if and only if three specific criteria are met. These criteria involve the monogenity of $f(x)$, divisibility conditions related to primes dividing $k$, and the squarefreeness of $f(0)$. Our results extend beyond the scope of Dedekind Criterion where the entire polynomial $f(x^k)$ is to be factored. %We prove that the monogenity of a monic irreducible power compositional polynomial of the type $f(x^k)$ is completely determined by the monogenity of $f(x)$, $f(0)$ and whether $p$ divide index of $f(x^p)$ for all primes $p\mid k$.\\  %One can check that the non-monogenity of $f(x)$ implies the non-monogenity of $f(x^k)$ for any positive integer $k$. But the converse need not to be true. In this paper, we prove that 

 %Let  $K=\Q(\theta)$ be an algebraic number field with $\theta$ satisfying a monic irreducible polynomial $f(x)$ of degree $n$ over $\Q.$ % An important problem in algebraic number theory is to determine whether $\{1,\theta,\cdots,\theta^{n-1}\}$ is an integral basis of $K$ or not. If $\{1,\theta,\cdots,\theta^{n-1}\}$ is an integral basis of $K$, then the polynomial $f(x)$ is said to be monogenic.To check whether a given monic irreducible polynomial is monogenic or not is a recent center of study in the field of algebraic number theory.\\
 
 \indent Let $k$ be a positive integer and $q$ be a prime divisor of $k$. Suppose $f(x)\in\Z[x]$ is a monic polynomial. If $f(x^k)$ is irreducible, then the roots $\alpha$ and $\beta$ of $f(x^q)$ and $f(x^k)$ respectively, give rise to a tower of fields as shown below.
\begin{figure}[h!]
 	\centering
 		\begin{tikzpicture}[scale=0.4]
 			\draw[thick, ->] (0, 0) -- (0, 1);
 			\draw[thick, ->] (0, 3) -- (0, 4);
 			\draw (0,-1) node{$\Q$};
    \draw (0,2) node{$\Q(\alpha)$};
    \draw (0,5) node{$\Q(\beta)$};
 		\end{tikzpicture}
 		\label{figure3}
 	\end{figure}\\
\indent Note that $\beta^{\frac{k}{q}}=\alpha$. In this paper, we discuss the relation between the indices $[\Z_L:\Z[\beta]]$ and $[\Z_K:\Z[\alpha]]$, where $L=\Q(\beta)$ and $K=\Q(\alpha)$, i.e., the relation between the index of $f(x^q)$ and the index of $f(x^k),$ where $q\mid k$. Moreover, we give necessary and sufficient conditions for these indices to be equal to $1$. \\
%\indent Recall that the discriminant of a monic polynomial over a field $\F$ of degree $n$ having roots $\theta_1, \cdots, \theta_n$ in the algebraic closure of $\F$ is defined to be $\triangle_f=\displaystyle\prod_{1\leq i<j\leq n}^{}(\theta_i-\theta_j)^2.$ %For an algebraic number field $K=\Q(\theta)$ with $\theta$ an algebraic integer having minimal polynomial $f(x)$ over  $\Q$, $\ind\theta$ will stand for the group index $[A_K:\Z[\theta]]$.
\indent  It is well known, that if $f(x)$ is the minimal polynomial of an algebraic integer $\theta$ over $\Q$, then  the discriminant $ D(f)$ of $f(x)$ and the discriminant $\Delta(K)$ of $K=\Q(\theta)$ are related by the formula 
\begin{equation}\label{eq}
D(f)=[\Z_K:\Z[\theta]]^2\cdot \Delta(K).
\end{equation} 
Clearly, if $D(f)$ is squarefree, then $f(x)$ is monogenic. But the converse is not always true.  Several classes of monogenic polynomials with non squarefree discriminants are known and their properties have been studied (cf.  \cite{Gaal} - \cite{moo},
\cite{LJ} - \cite{susu}, \cite {HS1}).  % A monic irreducible polynomial $g(x)\in\Z[x]$  is monogenic when discriminant $D(g)$ of $g(x)$ is square free. In case $D(g)$ is not square free;  the answer  of monogenity of the polynomial $g(x)$ is not trivial. In the last few decades,  numerous  classes of monogenic polynomials with non-square free discriminant have been known. 
L. Jones in \cite{ljj} studied the monogenity of power compositional Shanks polynomials. 
%ite{LJJ}, L. Jones  showed that there exist infinitely many primes $p \geq  3$, and integers $t \geq 1$ coprime to $ p$, such that $f(x) = x^p - 2ptx^{p-1} + p^2t^2x^{p-2} + 1$ is non-monogenic. In \cite{LJ1}, L. Jones  gave infinite families of monogenic polynomials  using new discriminant formula.  He conjectured that if $\gcd(n,mB)=1$  and $A$ is  prime, then  an irreducible polynomial of the type $f(x)=x^n+A(Bx+1)^m\in\Z[x]$ with $n\ge 3$ and $1\le m\le n-1,$ is monogenic if and only if $n^n+(-1)^{n+m}B^n(n-m)^{n-m}m^mA$ is square-free (cf.  \cite[Conjecture 4.1]{LJ}). In \cite{susu}, we proved that this  conjecture is true.   In the present paper, we look at a possible extensions of L. Jones' conjecture for an irreducible polynomials of the type $f(x)=x^n+A(Bx^k+1)^m\in \Z[x]$, when $n=kd, ~d>m\geq 1$, $B\ne0$, $\gcd(d,mBk)=1$ and $\gcd(k,d^d+(-1)^{d+m}B^dm^m(d-m)^{d-m}A)=1.$  \\

\indent  Throughout the paper, the irreducibility of a monic polynomial belonging to $\Z[x]$ is meant over $\Q$, unless stated otherwise. Note, that for a positive integer $k$, the irreducibility of a  monic polynomial $f(x^k)$ implies the irreduciblity of $f(x)$. For a positive integer $z,$ the radical of $z$, denoted by $\rad(z)$, is defined  as the largest squarefree integer dividing $z$.\\
\indent Next, we state our main theorems.
\begin{theorem}\label{sufnec}
    Let $f(x)$ be a monic polynomial with integer coefficients. Let $k\geq 2$ be an integer such that $f(x^k)$ is irreducible. Then $f(x^k)$ is monogenic if and only if 
    \begin{enumerate}
        \item $f(x)$ is monogenic,
		\item $p$ does not divide the index of $f(x^p)$ for all primes $p\mid k$ and
		\item $f(0)$ is squarefree.
    \end{enumerate} 
\end{theorem}

%\begin{theorem}\label{TH1}Let $f(x)$ be a monic polynomial having integer coefficients. Let $k$ be a positive integer such that $f(x^k)$ is irreducible. If $f(x^k)$ is monogenic, then\begin{enumerate}\item $f(x)$ is monogenic and\item $q$ does not divide the index of $f(x^q)$ for any prime $q\mid k$.\end{enumerate}\end{theorem}
\begin{remark}\label{TH1}
	Note, that for a monic polynomial $f(x)\in\Z[x]$, if the irreducible polynomial  $f(x^k)$ is monogenic, then by using Lemma \ref{L1} it is easy to conclude that $f(x^t)$ is also monogenic for each divisor $t$ of $k$.  
\end{remark}

\begin{remark}
    It is important to mention that the above theorem shows that the monogenity of $f(x^k)$ does not depend on the exponents of the primes in the factorization of $k$. This is a big difference compared to the Dedekind criterion, where the entire polynomial $f(x^k)$ is to be factored, not just $f(x^p)$. Furthermore, if $p$ does not divide the index of $f(x^p)$, then together with the points $(1)$ and $(3)$, this immediately implies the monogenity of the infinite family of polynomials $f(x^{p^u})$, $(u\in\N)$.
\end{remark}

\begin{comment} 
The following corollary can be quickly deduced from the above theorem.
%\begin{corollary}Let $k$ be a positive integer and $f(x)\in\Z[x]$ be a monic  polynomial such $f(x^k)$ is irreducible. If $f(x)$ is non-monogenic, then so is $f(x^k)$.\end{corollary}
%\begin{theorem}Let $k$ be a composite positive integer. Let $f(x)\in\Z[x]$ be a monic polynomial such that for each divisor $t$ of $k$, $f(x^t)$ is irreducible. If for each proper divisor $t$ of $k$, $f(x^t)$ is monogenic, then $f(x^k)$ is monogenic.\end{theorem}\begin{proof} Let $\theta$ be a root of $f(x^k).$ Then $\theta^k$ is a root of $f(x).$ Suppose $f(x^k)$ is non-monogenic. Then there exists a prime $p$ dividing the index of $f(x^k).$ So for certain polynomial $h(x)\in\Z[x],$ we see that $\frac{h(\theta)}{p}$ is an algebraic integer belonging to $\Q(\theta)$.\end{proof}

\begin{corollary}\label{Cc1}
	Let $f(x)$ be a monic polynomial having integer coefficients. Let $k$ be a positive integer such that $f(x^k)$ is irreducible. If $f(x^k)$ is monogenic, then
	\begin{enumerate}
		\item $f(x)$ is monogenic and
		\item $p$ does not divide the index of $f(x^p)$ for any prime $ p\mid k$.
	\end{enumerate}
\end{corollary}
\begin{remark}
	It may be pointed out that the converse of the above corollary is not true in general. For example, consider $f(x)=x^2+x+4$. Then $f(x)$ is monogenic
and $3$ does not divide the index of $f(x^3)$, 
but $f(x^3)$ is non-monogenic because $2$ divides the index of $f(x^3)$. 
\end{remark}
\end{comment}

In the proof of Theorem \ref{sufnec} we will see that if $f(x)$ is monogenic, then only the prime divisors of $k$ and $f(0)$ can divide the index of $f(x^k)$, so we obtain the following corollary to be used in the proof of Theorem \ref{abc}. 

\begin{corollary}\label{TH2}
	Let $f(x)$ be a monic polynomial having integer coefficients. Let $k$ be a positive integer such that $f(x^k)$ is irreducible. Suppose, that
	\begin{enumerate}
		\item $f(x)$ is monogenic and
		\item $p$ does not divide the index of $f(x^p)$ for any prime $p\mid k$.
	\end{enumerate} Then every prime divisor of the index of $f(x^k)$ will divide $f(0)$.
\end{corollary}

In view of the above corollary, if $|f(0)|=1$ or the divisors of $f(0)$ do not divide the index of $f(x^k)$, then (1) and (2) together are equivalent to the monogenity of $f(x^k)$. This happens, for example, in the case of polynomials that are Eisenstein with respect to any prime divisor of $f(0)$, see Lemma \ref{th2.1}. These polynomials can be given of the form  $f(x)=x^d+A\cdot h(x)$, where $|h(0)|=1$ and $A$ is squarefree. Adding a further condition, the monogenity of $f(x)$ will be equivalent to the monogenity of $f(x^k)$.

\begin{corollary}\label{abcde}
     Let $f(x)\in\Z[x]$ be a polynomial of the form $f(x)=x^d+A\cdot h(x)$, where $d>1, |h(0)|=1$ and $\deg h(x)<d$. Let $k$ be a positive integer such that $f(x^k)$ is irreducible. If $\rad(k)$ divides $\rad(A)$, then $f(x)$ is monogenic if and only if $f(x^k)$ is monogenic.
 \end{corollary}
Note, that Theorem \ref{sufnec} implies that if $f(x^{p^u})$ is monogenic, then $f(x)$ is monogenic and $f(0)$ is squarefree.  In the following theorem we prove, that its  converse is true with an additional requirement. This theorem suggests that the monogenity of $f(x^{p^u})$ can be directly obtained from the monogenity of $f(x^p)$. More precisely, we give sufficient conditions for the monogenity of a monic irreducible prime power compositional polynomial. 
\begin{theorem}\label{uvw}
    Let $u$ be a positive integer, $p$ a prime number and $f(x)\in\Z[x]$ a monic polynomial. Suppose that $f(x^{p^u})$ is irreducible. Then the following statements hold:
    \begin{enumerate}
        \item If $f(x)$ is monogenic and $f(0)$ is squarefree, then the index of $f(x^{p^u})$ is equal to $p^s$, with some non-negative integer $s$. 
        \item $p$ does not divide the index of $f(x^{p^u})$ if and only if $\frac{1}{p}(f(x^p)-f(x)^p)$ is coprime to $f(x)$ modulo $p$.
        \end{enumerate}
\end{theorem}
\begin{remark}\label{remcop}
Note, that if $u=1$ in the above theorem, then the condition (2) of Theorem \ref{sufnec} is equivalent to the condition that $\frac{1}{p}(f(x^p)-f(x)^p)$ is coprime to $f(x)$ modulo $p$, for every prime $p\mid k$.
\end{remark}

Combining the Corollary \ref{abcde} and Theorem \ref{uvw} we get the statement below, which also follows from Lemma 3.1 of L. Jones \cite{ljj1}.

\begin{corollary}\label{xyz}
    Let $p$ be a prime number. Suppose, that $f(x)\in\Z[x]$ is a monic and Eisenstein polynomial with respect to $p$. If $f(0)$ is squarefree, then $f(x)$ is monogenic if and only if $f(x^{p^u})$ is monogenic for any positive integer $u$. 
\end{corollary}

Taking Remark \ref{remcop} into account, when $f(x)=x^d+A\cdot h(x)\in\Z[x]$, with $|h(0)|=1$ and $\deg h(x)<d$, we can provide necessary and sufficient conditions for the monogenity of the irreducible power compositional polynomial $f(x^k)$. As an application of this theorem, we constructed an infinite family of polynomials for a specific choice of $h(x)$, given in the last section.
    \begin{theorem}\label{abc} Let $f(x)\in\Z[x]$ be a polynomial of the form $f(x)=x^d+A\cdot h(x)$, where $d>1, |h(0)|=1$ and $\deg h(x)<d$. Let $k$ be a positive integer such that $f(x^k)$ is irreducible. Then $f(x^k)$ is monogenic if and only if 
	\begin{enumerate}
		\item $f(x)$ is monogenic, and
		\item $\frac{1}{p}(A\cdot h(x^p)+(-A\cdot h(x))^p)$ is coprime to $f(x)$ modulo $p$, for any prime $p\mid k$.
	\end{enumerate}   \end{theorem}
\section{Preliminary results}
In what follows, for a prime $p$ and a polynomial $g(x) \in \Z[x],~ \overline{g} (x)$ will
denote the polynomial obtained by reducing each coefficient of $g(x)$ modulo $p$. For a subset $S\subset \Z[x]$, let $\langle S\rangle$ be the ideal generated by $S$. If $S=\{s_1,s_2,\ldots,s_k\}$ is finite, then 
$$\langle S\rangle=\{s_1a_1+s_2a_2+\ldots+s_ka_k\mid a_i\in\Z[x],~1\leq i\leq k\}.$$

To investigate the monogenity of $f(x)$, we will follow the approach of K. Uchida, so we recall Lemma 2.A and Theorem 2.B proved in \cite{UC}.

\noindent \textbf{Lemma 2.A.}
An  ideal $\mathfrak M$ of $\mathbb{Z}[x]$ containing a monic polynomial is maximal if and only if $\mathfrak M =  \langle p, g(x) \rangle $  for some prime number   $p$  and a monic polynomial  $g(x)$ belonging to $\mathbb Z [x]$ which is irreducible modulo $p$.

%\begin{remark}Since for any $f(x)\in\Z[x]$, $f(x^p)-f(x)^p$ is contained in $\langle p\rangle\subset\langle p,g(x)\rangle$, and the maximal ideal $\langle p, g(x) \rangle $ is also a prime ideal, then we obtain$$f(x^p)\in \langle p, g(x) \rangle \Longleftrightarrow f(x)^p \in \langle p, g(x) \rangle \Longleftrightarrow f(x)\in \langle p, g(x) \rangle .$$\end{remark}

\noindent \textbf{Theorem 2.B.}
Let  $K=\mathbb Q(\theta)$  with  $\theta$ in $\Z_K$ having  minimal polynomial  $f(x)$   over   $\mathbb Q$,  then $\Z_K=\Z[\theta]$ if and only if $f(x)$ does not belong to $\mathfrak M ^2$ for any maximal ideal $\mathfrak M$ of the polynomial ring $\mathbb Z[x]$.

In particular, a prime $p$ divides the index $[\Z_K:\Z[\theta]]$ if and only if $f(x)$ belongs to the square of a maximal ideal of the form $\mathfrak M=\langle p, g(x) \rangle$, where $g(x)\in\Z[x]$ is a monic polynomial which is irreducible modulo $p$. We can split the condition $f(x)\in\langle p, g(x) \rangle^2$ into two parts. 

\begin{lemma}\label{maxint}
    Let $p$ be a prime and $g(x)\in\Z[x]$ be a monic polynomial which is irreducible modulo $p$. Then
    $\langle p,g(x) \rangle^2=\langle p^2,g(x)\rangle \cap \langle p,g^2(x)\rangle.$
\end{lemma}

\begin{proof}
    Clearly,
    $\langle p,g(x) \rangle^2\subseteq\langle p^2,g(x)\rangle \cap \langle p,g^2(x)\rangle$. To obtain the other direction of the containment, suppose that the polynomial $f(x)$ belongs to both $\langle p^2,g(x)\rangle$ and $\langle p,g^2(x)\rangle$, i.e., there exist $a(x),b(x),c(x),d(x)$ belonging to $\Z[x]$, such that
    $$p^2\cdot a(x)+g(x)\cdot b(x)=f(x)=p\cdot c(x)+g^2(x)\cdot d(x).$$
    Since $g(x)$ is irreducible modulo $p$, it is irreducible in $\Z[x]$. So $\langle g(x)\rangle$ is a prime ideal of $\Z[x]$. By the above equation, $p\cdot (c(x)-p\cdot a(x))$ is a multiple of $g(x)$, i.e., it belongs to the prime ideal $\langle g(x)\rangle$. Obviously, $p$ does not belong to $\langle g(x)\rangle$, thus $c(x)-p\cdot a(x)\in\langle g(x)\rangle$ and therefore $c(x)\in \langle p, g(x)\rangle$. This implies that $p\cdot c(x) \in \langle p^2, p\cdot g(x)\rangle$, and from there
    $f(x)=p\cdot c(x)+g^2(x)\cdot d(x) \in \langle p^2, p\cdot g(x),g^2(x)\rangle$. Hence $f(x)\in \langle p,g(x) \rangle^2.$
\end{proof} 

    If $f(x)\in \langle p,g^2(x)\rangle$, then $\overline{g}(x)$ is a multiple factor of $\overline{f}(x)$. On the other hand, $f(x)\in \langle p^2,g(x)\rangle$ implies that the remainder of $f(x)$ divided by $g(x)$ is a multiple of $p^2$. Therefore, the results of Uchida \cite{UC} say that $p$ divides the index of $f(x)$ if and only if there exists a monic polynomial $g(x)\in\Z[x]$, which is a multiple irreducible factor of $f(x)$ modulo $p$, such that the remainder of $f(x)$ divided by $g(x)$ is a multiple of $p^2$. This also follows from the results of Ore \cite{ore} and from the Dedekind Criterion.

\begin{corollary}\label{resqu}
    Let $f(x)$ and $g(x)$ belonging to $\Z[x]$ be monic polynomials and let $p$ be a prime number such that $g(x)$ is irreducible modulo $p$. Suppose that $f(x)\in\langle p,g^2(x) \rangle$, then
    $$f(x)\in\langle p, g(x) \rangle^2\iff  f(x)\in\langle p^2, g(x) \rangle.$$
\end{corollary}

Considering a prime divisor $p$ of $f(0)$, it is easy to see that $x$ is a factor of $f(x)$ modulo $p$ and it is a multiple factor of $f(x^\ell)$ modulo $p$ if $\ell>1$. So, by using the above corollary, we obtain the following result.

\begin{proposition}\label{re}
    Let $f(x)\in\Z[x]$ be a monic polynomial and $p$ be a prime, such that $p^2\mid f(0)$. Then $p$ divides the index of $f(x^\ell)$ for any positive integer $\ell>1$.
\end{proposition}
\begin{proof}
    Since $p^2$ divides $f(0)$, we can see that $x$ is a multiple factor of $f(x^\ell)$ modulo $p$ for any integer $\ell>1$. Therefore $f(x^\ell)\in\langle p,x^2\rangle$. Moreover, the remainder of $f(x^\ell)$ divided by $x$ is $f(0)$, which is a multiple of $p^2$, thus $f(x^\ell)\in \langle p^2,x\rangle$. By Corollary \ref{resqu} and Theorem 2.B,  $p$ divides the index of $f(x^\ell)$.
\end{proof} 

In the following theorem, we show that if $\bar{g}(x)$ is a multiple irreducible factor of $\bar{f}(x)$, then the remainder of $f(x)$ divided by $g(x)$ is a multiple of $p^2$ if and only if the remainder of $f(x^p)$ divided by $g(x)$ is a multiple of $p^2$. This observation is important for the proof of our main results and it is of independent interest as well. This allows us to ignore the exponents of the primes in the factorization of $k$ and to extend the monogenity to infinite families of polynomials of unlimited degree.

\begin{theorem}\label{T1}
    Let $f(x)$ and $g(x)$ belonging to $\Z[x]$ be monic polynomials and let $p$ be a prime number such that $g(x)$ is irreducible modulo $p$. Suppose that $f(x)\in\langle p,g^2(x) \rangle$. Then 
    $$f(x)\in\langle p, g(x) \rangle^2\iff f(x^p)\in\langle p, g(x) \rangle^2.$$
\end{theorem}

\begin{proof} Write 
    $f(x)=g^2(x)\cdot q(x)+p\cdot r(x)$
    for some $q(x),r(x)$ belonging to $\Z[x]$. Let $r(x)=g(x)\cdot s(x)+t(x)$, where $\deg t(x)<\deg g(x)$. Then
    \begin{equation}\label{fgp}
        f(x)=g^2(x)\cdot q(x)+p\cdot g(x)\cdot s(x)+p\cdot t(x).
    \end{equation}
    Since $g^2(x)\in \langle p, g(x) \rangle^2$ and $p\cdot g(x)\in \langle p, g(x) \rangle^2$, we get that
    \begin{equation}\label{eqq1}
    f(x)\in\langle p, g(x) \rangle^2 \iff p\cdot t(x)\in\langle p, g(x) \rangle^2.
    \end{equation}
\begin{comment}
    (1) Since $g^2(x)\in \langle p, g(x) \rangle^2$ and $p\cdot g(x)\in \langle p, g(x) \rangle^2$, then \begin{equation}\label{eqq1}
        f(x)\in\langle p, g(x) \rangle^2 \iff p\cdot t(x)\in\langle p, g(x) \rangle^2.\end{equation} Since $\deg t(x)<\deg g(x)$, we get $f(x)\in\langle p, g(x) \rangle^2$ if and only if $p\cdot t(x)\in\langle p^2\rangle$. Clearly $p\cdot t(x)$ is the remainder of $f(x)$ divided by $g(x)$, and this remainder belongs to $\langle p^2\rangle$ if and only if $f(x)\in\langle p^2, g(x) \rangle$.\\
\end{comment}    
    Substituting $x^p$ in place of $x$ in \eqref{fgp},
    \begin{equation}\label{eqq2}f(x^p)=g^2(x^p)\cdot q(x^p)+p\cdot g(x^p)\cdot s(x^p)+p\cdot t(x^p).\end{equation}
    Using the fact, that $p\mid (g(x^p)-g(x)^p)$ we get $g(x^p)\in\langle p, g(x) \rangle$. This implies that both $g^2(x^p)$ and $p\cdot g(x^p)$ belong to the ideal $\langle p, g(x) \rangle^2$. Thus, $f(x^p)\in\langle p, g(x) \rangle^2$ if and only if $p\cdot t(x^p)\in \langle p, g(x) \rangle^2$. Since $p\mid (t(x)^p-t(x^p))$ for any polynomial $t(x)\in\Z[x]$, then \begin{equation}\label{eqq3}p\cdot t(x^p)\in \langle p, g(x) \rangle^2 \iff p\cdot t(x)^p\in \langle p, g(x) \rangle^2.\end{equation} %To sum it up, by the previous point,$$f(x)\in\langle p, g(x) \rangle^2 ~ \text{ if and only if } ~ p\cdot t(x)\in \langle p, g(x) \rangle^2,$$
    Thus, from \eqref{eqq2} and \eqref{eqq3},
    \begin{equation}\label{eqq4} f(x^p)\in\langle p, g(x) \rangle^2 ~ \iff p\cdot t(x)^p\in \langle p, g(x) \rangle^2.\end{equation}
    In view of \eqref{eqq1} and \eqref{eqq4}, to prove the statement it is enough to  show that 
    $$p\cdot t(x)\in \langle p, g(x) \rangle^2\text{ if and only if }p\cdot t(x)^p\in \langle p, g(x) \rangle^2.$$ 
    Clearly, $p\cdot t(x)\in \langle p, g(x) \rangle^2$ implies that $p\cdot t(x)^p\in \langle p, g(x) \rangle^2$. Assuming that $p\cdot t(x)^p\in \langle p, g(x) \rangle^2$, we can write
    $$p\cdot t(x)^p=p^2\cdot a(x)+p\cdot g(x)\cdot b(x)+g^2(x)\cdot c(x),$$
    for some $a(x),b(x),c(x)\in\Z[x].$
    Therefore, $g^2(x)\cdot c(x)$ belongs to the prime ideal $\langle p\rangle$ of $\Z[x]$. Since $g(x)$ is monic and irreducible modulo $p$, then $c(x)\in \langle p\rangle$ and $t(x)^p\in \langle p, g(x) \rangle$. Equivalently, $t(x)\in \langle p, g(x) \rangle$ and therefore $p\cdot t(x)\in \langle p, g(x) \rangle^2$. This completes the proof.
\end{proof}

\begin{corollary}\label{C1}
	Let $u$ be a positive integer. For a monic polynomial $f(x)\in\Z[x]$, a prime $p$ divides the index of $f(x^p)$ if and only if $p$ divides the index of $f(x^{p^u})$, provided that $f(x^{p^u})$ is an irreducible polynomial.
\end{corollary}
\begin{proof}
    Let $g(x)\in\Z[x]$ be a monic polynomial such that $\overline{g}(x)$ is an irreducible factor of $f(x)$ modulo $p$. As $f(x^p)\equiv f(x)^p\mod p$, $g(x)^p$ is a factor of $f(x^p)$ modulo $p$, i.e., $f(x^p)\in\langle p, g^2(x) \rangle$. By Theorem \ref{T1}, 
    $$f(x^p)\in\langle p, g(x) \rangle^2 \iff  f(x^{p^2})\in\langle p, g(x) \rangle^2.$$
    Applying induction on $u$,  we get
    $$f(x^p)\in\langle p, g(x) \rangle^2 \iff f(x^{p^u})\in\langle p, g(x) \rangle^2.$$
    Keeping Theorem 2.B in mind, $p$ divides the index of $f(x^p)$ if and only if $p$ divides the index of $f(x^{p^u})$.
\end{proof}
%\begin{corollary}Let $u$ be a positive integer and $f(x^{p^u})\in\Z[x]$ be a monic irreducible polynomial. Then $f(x^p)$ is monogenic  if and only if $f(x^{p^u})$ is monogenic.    \end{corollary}
\begin{lemma}\label{L1}
    Let $f(x)\in\Z[x]$ be a monic polynomial and $\ell$ a positive integer such that $f(x^\ell)$ is irreducible. If a prime $p$ divides the index of $f(x)$, then $p$ divides the index of $f(x^\ell)$. 
\end{lemma}
\begin{proof}
    If $p$ divides the index of $f(x)$, then by Theorem 2.B, there exists a monic polynomial $g(x)\in\Z[x]$, such that $g(x)$ is irreducible modulo $p$ and $f(x)\in \langle p, g(x) \rangle^2$. Therefore, $f(x^\ell)\in \langle p, g(x^\ell) \rangle^2$. Let $h(x)$ be an irreducible factor of $g(x^\ell)$ modulo $p$. Then $g(x^\ell)\in \langle p,h(x)\rangle$, so $\langle p, g(x^\ell) \rangle\subset \langle p,h(x)\rangle$. This implies, that $f(x^\ell)$ belongs to the square of the maximal ideal $\langle p, h(x) \rangle$. Hence, by Theorem 2.B, $p$ divides the index of $f(x^\ell)$.
\end{proof}
To investigate all of the possible primes that may divide the index of $f(x^\ell)$, we need to find its discriminant. We shall use the following lemma, which describes the properties of the resultant of two polynomials.
\begin{lemma}\label{lemma a} Let $f(x)=a_mx^m+a_{m-1}x^{m-1}+\cdots+a_0$ and $g(x)=b_nx^n+b_{n-1}x^{n-1}+\cdots+b_0$  be two polynomials with integer coefficients of degrees $m$ and  $n$, and let $\lambda_1,\cdots,\lambda_m$ and $\mu_1\cdots,\mu_n$ be the roots of $f(x)$ and $g(x)$,  respectively. Then
\begin{enumerate}
    \item $R(f,g)=a_m^nb_n^m\displaystyle\prod_{\substack{1\le i\le m\\ 1\le j \le n}}(\lambda_i-\mu_j)$,
	\item $R(f,g)=a_m^n\prod\limits_{i=1}^m g(\lambda_i)=(-1)^{mn}b_n^m\prod\limits_{j=1}^n f(\mu_j)$,
\item $R(f,gh)=R(f,g)R(f,h)$, for any $h(x)\in\Z[x]$,
\item $R(f,ag)=a^mR(f,g)$, for any $a\in\Z$.
\end{enumerate}
\end{lemma}

Note, that the discriminant of a monic polynomial $f(x)\in\Z[x]$ of degree $n$ and the resultant $R(f,f')$ are related by the following formula.
\begin{equation}\label{beta}
	D(f)=(-1)^{\frac{n(n-1)}{2}}R(f,f'),
\end{equation} where $f'$ is the derivative of $f.$ 
\begin{lemma}\label{dfpow}
	Let $f(x)\in\Z[x]$ be a monic polynomial of degree $d$, and $\ell$ a positive integer. Then the discriminant of $f(x^\ell)$ is
	$$D(f(x^\ell))=\pm D(f(x))^\ell\cdot \ell^{d\ell}\cdot f(0)^{\ell-1}.$$
\end{lemma}
\begin{proof} Using \eqref{beta}, we get
		$$D(f(x^\ell))=(-1)^{\frac{d\ell(d\ell-1)}{2}}\cdot R(f(x^\ell),\ell x^{\ell-1}f'(x^\ell)).$$
		Applying (3) and (4) of Lemma \ref{lemma a}, 
        $$D(f(x^\ell))=(-1)^{\frac{d\ell(d\ell-1)}{2}}\cdot \ell^{d\ell}\cdot R(f(x^\ell),x^{\ell-1})\cdot R(f(x^\ell),f'(x^\ell)).$$ 
        Using (2) of Lemma \ref{lemma a}, 
		\begin{align*}D(f(x^\ell))=&(-1)^{\frac{d\ell(d\ell-1)}{2}+d\ell+\ell-1}\cdot \ell^{d\ell}\cdot\prod_{i=1}^{\ell-1}f(0)\cdot R(f(x^\ell),f'(x^\ell)),\\
		=&(-1)^{\frac{d\ell(d\ell+1)}{2}+\ell-1}\cdot  \ell^{d\ell}\cdot f(0)^{\ell-1}\cdot R(f(x^\ell),f'(x^\ell)),\\
		=&(-1)^{\frac{d\ell(d\ell+1)}{2}+\ell-1}\cdot  \ell^{d\ell}\cdot f(0)^{\ell-1}\cdot \prod_{i=1}^{d\ell}f'(\lambda_i^\ell),\end{align*}
        where $f(\lambda_i^\ell)=0$, for $1\le i\le d\ell.$	 Note, that if $\mu_1,\cdots, \mu_d$ are roots of $f(x),$ then we can assume that $\lambda_{i+j-1}^\ell=\mu_j$ for $1\le i\le \ell$ and $1\le j\le d.$ Therefore, (1) and (2) of Lemma \ref{lemma a} imply that 
        \begin{align*}
		D(f(x^\ell))=&(-1)^{\frac{d\ell(d\ell+1)}{2}+\ell-1}\cdot  \ell^{d\ell}\cdot f(0)^{\ell-1}\cdot\left(\prod_{i=1}^{d}f'(\mu_i)\right)^\ell.
		%=&(-1)^{\frac{d\ell(d\ell+1)}{2}+\ell-1}\cdot   \ell^{d\ell}\cdot f(0)^{\ell-1}\cdot \left(D(f(x))\right)^\ell
        \end{align*}
        Hence, the result follows from (2) of Lemma \ref{lemma a}.
\end{proof}
%This result shows that the possible primes that may divide the index of $f(x^\ell)$ are the prime divisors of $D(f)\cdot \ell\cdot f(0)$. 
The following lemma is proved in \cite[Proposition 4.15]{NW}, and it suggests that if $f(x)$ is monogenic, then the discriminant of the field generated by a root of $f(x^\ell)$ cannot be too small.
\begin{lemma}\label{abcd}
	Let $K$ and $L$ be number fields with $K\subset L$. Then $$\Delta(K)^{[L:K]}\mid \Delta(L).$$
\end{lemma}

The next proposition is essential for the proof of our main theorems. It completely describes the possible prime divisors of the index of $f(x^\ell)$.

\begin{proposition}\label{C2}
    Let $f(x)\in\Z[x]$ be a monic  polynomial of degree $d$, $p$ a prime and $\ell$ an integer greater than $1$. If  $p\nmid \ell$ and $f(x^\ell)$ is irreducible, then $p$ divides the index of $f(x^\ell)$ if and only if either 
    \begin{enumerate}
        \item $p$ divides the index of $f(x)$, or
        \item $p^2\mid f(0)$.
    \end{enumerate}
\end{proposition}
\begin{proof} 
    By Lemma \ref{L1} and Proposition \ref{re}, both (1) and (2) imply that $p$ divides the index of $f(x^\ell)$. To prove the converse,
     we assume that $p$ does not divide the index of $f(x)$ and $p^2\nmid f(0)$. We split the proof into two cases.\\
    \textit{Case 1.} $p\nmid f(0)$. Equation \eqref{eq} shows, that 
    $$[\Z_K:\Z[\theta]]^2=\frac{D(f(x))}{\Delta(K)},$$
    where $\Delta(K)$ is the discriminant of the number field $K$ generated by a root $\theta$ of $f(x)$. Let $L$ be a number field generated by a root of $f(x^\ell)$. From Lemma \ref{dfpow},
    $$D(f(x^\ell))=\pm D(f(x))^\ell\cdot \ell^{d\ell}\cdot f(0)^{\ell-1}.$$ 
    Using Lemma \ref{abcd}, we get
    $$\frac{D(f(x^\ell))}{\Delta(L)}\mid \frac{D(f(x))^\ell}{\Delta(K)^\ell}\cdot \ell^{d\ell}\cdot f(0)^{\ell-1},$$ where $\Delta(L)$ is the discriminant of the number field $L$.
    Since $p\nmid \frac{D(f(x))}{\Delta(K)}$ and $p\nmid \ell\cdot f(0)$, then $p$ does not divide the index of $f(x^\ell)$.\\
    \textit{Case 2.} $p\mid f(0)$. In view of Theorem 2.B, we have to show that $f(x^\ell)$ does not belong to $\langle p,h(x)\rangle^2$, where $h(x)\in\Z[x]$ is monic and irreducible modulo $p$. By Lemma \ref{maxint},
    %$$\langle p,h(x)\rangle^2=\langle p,h^2(x)\rangle \cap \langle p^2,h(x)\rangle,$$
    it is enough to show that $f(x^\ell)$ does not belong to either $\langle p,h^2(x)\rangle$ or $\langle p^2,h(x)\rangle$. %, i.e., either $h(x)$ is not a multiple factor of $f(x^\ell)$ modulo $p$ or the remainder of $f(x^\ell)$ modulo $h(x)$ is not a multiple of $p^2$.\\    
   If $h(x)=x$, then $f(x^\ell)-f(0)\in \langle x\rangle\subset\langle p^2,x\rangle$. But $f(0)\not\in\langle p^2,x\rangle$, so $f(x^\ell)\not\in \langle p^2,x\rangle$. %This means that we can ignore the multiple factor $x$ of $f(x^\ell)$ modulo $p$.
    %Therefore, we need multiple factor of $f(x^\ell)$ modulo $p$. 
    First, we show that $f(x^\ell)$ has a multiple irreducible factor modulo $p$ different from $x$ if and only if $f(x)$ has a multiple irreducible factor modulo $p$ different from $x$.
    
    Let $g(x)$ be a monic irreducible polynomial modulo $p$, then $\gcd(g(x),g'(x))=1$. Since the derivative of $g(x^\ell)$ is $\ell x^{\ell-1}g'(x^\ell)$ and $p\nmid \ell$, $g(x^\ell)$ is separable if and only if $g(x)\neq x$. Thus, $\gcd(g(x^\ell),g'(x^\ell))=1$. We can conclude, that whenever a monic irreducible polynomial $g(x)\neq x$ is a single factor of $f(x)$ modulo $p$, then all irreducible factors of $g(x^\ell)$ are single factors of $f(x^\ell)$ modulo $p$. Furthermore, if $g_1(x)$ and $g_2(x)$ are distinct irreducible factors of $f(x)$ modulo $p$, then $\gcd(g_1(x),g_2(x))=1$ and so $\gcd(g_1(x^\ell),g_2(x^\ell))=1$. Thus, $h(x)\neq x$ is a multiple irreducible factor of $f(x^\ell)$ modulo $p$ if and only if $f(x)$ has a multiple irreducible factor $g(x)$ modulo $p$ such that $\overline{h}(x)\mid \overline{g}(x^\ell)$.

    Let $g(x)\in\Z[x]$ be a monic polynomial which is a multiple irreducible factor of $f(x)$ modulo $p$. Then $f(x)\in \langle p, g^2(x)\rangle$. By Corollary \ref{resqu} and Theorem 2.B, with the hypothesis that $p$ does not divide the index of $f(x)$, we get $f(x)\not\in \langle p^2, g(x)\rangle$. Let us write  $f(x)=g(x)\cdot c(x)+r(x)$, where $\deg r(x)<\deg g(x)$ and $c(x), r(x)\in\Z[x]$. %Since $g(x)$ is a multiple factor of $f(x)$ modulo $p$, we see that
    Clearly, $r(x)\in \langle p\rangle$ and $\overline{g}(x)\mid \overline{c}(x)$. Set $r(x)=p\cdot s(x)$, for some $s(x)\in\Z[x]$. Then, for certain polynomials $a(x),b(x)\in\Z[x]$, we can write
    $f(x)=g^2(x)\cdot a(x)+p\cdot g(x)\cdot b(x)+p\cdot s(x)$. Consequently,
    \begin{equation}\label{fgs}
        f(x^\ell)=g^2(x^\ell)\cdot a(x^\ell)+p\cdot g(x^\ell)\cdot b(x^\ell)+p\cdot s(x^\ell).
    \end{equation}
    As $g(x)$ is irreducible modulo $p$ and $\deg s(x)=\deg r(x)<\deg g(x)=\deg \overline{g}(x)$, then $\gcd(\overline{g}(x),\overline{s}(x))=1$. Thus, there exists a polynomial $t(x)\in\Z[x]$ such that $s(x)\cdot t(x)\in 1+\langle p,g(x)\rangle$ and so
    $s(x^\ell)\cdot t(x^\ell)\in 1+\langle p,g(x^\ell)\rangle.$
    Let $h(x)\in\Z[x]$ be a monic polynomial which is an irreducible factor of $g(x^\ell)$ modulo $p$, then $g(x^\ell)\in \langle p,h(x)\rangle$, and 
    \begin{equation}\label{sph}
        s(x^\ell)\cdot t(x^\ell)\in 1+\langle p,h(x)\rangle.
    \end{equation}
    Since $h(x)$ is a multiple factor of $f(x^\ell)$ modulo $p$, $f(x^\ell)\in\langle p,h^2(x)\rangle.$    
    We have to show that $f(x^\ell)\not\in\langle p^2,h(x)\rangle$. Since $g(x^\ell)\in \langle p,h(x)\rangle$, we can see that 
    $$g^2(x^\ell)\in \langle p,h(x)\rangle^2\subset\langle p^2,h(x)\rangle$$
    and 
    $$p\cdot g(x^\ell)\in \langle p^2,p\cdot h(x)\rangle\subset\langle p^2,h(x)\rangle.$$ Using Equation \eqref{fgs}, we get that $f(x^\ell)\in \langle p^2,h(x)\rangle$ if and only if $p\cdot s(x^\ell)\in\langle p^2,h(x)\rangle$. We can write $s(x^\ell)=h(x)\cdot u(x)+v(x),$ where $\deg v(x) <\deg h(x)$ and $u(x), v(x)\in\Z[x]$.
    Thus, $p\cdot s(x^\ell)\in \langle p^2,h(x)\rangle$ if and only if $p\cdot v(x)\in\langle p^2,h(x)\rangle$. As $\deg v(x)<\deg h(x)$, $p\cdot v(x)\in\langle p^2,h(x)\rangle$ if and only if $v(x)\in\langle p\rangle$. If $v(x)$ is a multiple of $p$, then $s(x^\ell)\in\langle p,h(x)\rangle$, which is a contradiction to Equation \eqref{sph}. Therefore, $f(x^\ell)\not\in \langle p^2,h(x)\rangle$. Hence, Theorem 2.B together with Lemma \ref{maxint} implies that $p$ does not divide the index of $f(x^\ell)$.
\end{proof}

\begin{remark}
    The main idea behind the last part of the proof is that if $g(x)$ is irreducible modulo $p$ and $\deg s(x)<\deg g(x)$, then $\overline{g}(x^\ell)$ and $\overline{s}(x^\ell)$ are coprime. But $f(x^\ell)$ belongs to $\langle p^2,h(x)\rangle$ if and only if $\overline{h}(x)$ is a common divisor of $\overline{g}(x^\ell)$ and $\overline{s}(x^\ell)$.
\end{remark}

\begin{remark}
    As it seems from the previous proof, the most challenging part is to show that if $p\mid f(0)$, but $p^2\nmid f(0)$ and $p\nmid k$, then $p$ does not divide the index of $f(x^k)$. This observation makes condition (3) in Theorem \ref{sufnec} to be sufficient together with the other conditions.
\end{remark}

%\begin{proposition}\label{C3}Let $f(x)\in\Z[x]$ be a monic polynomial such that $f(0)$ is squarefree. If $f(x^2)$ is irreducible, then an odd prime $p$ divides the index of $f(x)$ if and only if $p$ divides the index of $f(x^2)$. \end{proposition}\begin{proof}In view of the above proposition, let $p\mid f(0).$ If $p$ divides the index of $f(x),$ then by Lemma \ref{L1}, $p$ divides the index of $f(x^2).$ Conversely, suppose that $p$ does not divide the index of $f(x)$, then $$\frac{D(f(x^2))}{\Delta(L)}\mid \frac{D(f(x))^2}{\Delta(K)^2}\cdot 2^{2d}\cdot f(0).$$Since $f(0)$ is squarefree, using \eqref{eq}, we see that $p$ will not divide the index of $f(x^2).$  \end{proof} %\begin{remark}The above corollary does not hold if either $F(0)$ is not squarefree or $p=2.$\end{remark}
 The following lemma proved in \cite[Lemma 2.17]{NW} will be used in the proof of the Theorem \ref{abc}.
\begin{lemma}\label{th2.1}
    Let $\alpha$ be an algebraic integer, and let $L=\Q(\alpha)$ be the field generated by it. If the minimal polynomial $f(x)$ of $\alpha$ over $\Q$ is an Eisenstein polynomial with respect to the prime $p$, i.e., it has the form $x^n+a_{n-1}x^{n-1}+\cdots+a_0,$ with $a_0,\cdots,a_{n-1}$ divisible by $p$ and $p^2\nmid a_0,$ then the index of $f(x)$ is not divisible by $p.$
\end{lemma}

\section{Proof of the main results}
\begin{proof}[Proof of Theorem \ref{sufnec}]
    First, assume that $f(x^k)$ is monogenic. Then by Lemma \ref{L1}, $f(x)$ and $f(x^p)$ are also monogenic for all $p\mid k$, so (1) and (2) hold. Furthermore, there is no prime $p$ that divides the index of $f(x^k)$, so by Proposition \ref{re}, $p^2\nmid f(0)$ for any prime $p$, i.e., $f(0)$ is squarefree.
    
    Conversely, assume that (1), (2) and (3) hold. By Proposition \ref{C2}, if $p\nmid k$, then (1) and (3) together imply that $p$ does not divide the index of $f(x^k)$.
   Let $p\mid k$ and write $k=p^u\ell$, where $p\nmid \ell$. Using Corollary  \ref{C1} and (2), we see that $p$ does not divide the index of $f(x^{p^u})$. Now, using (3) and Proposition \ref{C2}, we get that $p$ does not divide the index of $f(x^k)$. Thus, $f(x^{k})$ is monogenic.
\end{proof}
\begin{comment}
\begin{proof}[Proof of Theorem \ref{TH1}] Clearly the irreducibility of the polynomial $f(x^k)$ implies that the polynomial $f(x^m)$ is irreducible for each divisor $m$ of $ k.$
	Using  Lemma \ref{L1}, we can see that for any divisor $t$ of $k$, if a prime $p$ divides the index of $f(x^t)$, then it will divide the index of $f(x^k)$. Since $f(x^k)$ is monogenic, we get that the polynomial $f(x^t)$ is monogenic for each divisor $t$ of $k.$
\end{proof}
\end{comment}

\begin{proof}[Proof of Corollary \ref{TH2}] Let $p$ be a prime divisor of the index of $f(x^k)$. Suppose, that $p\nmid  f(0)$. If $p\nmid k$, then by Proposition \ref{C2}, $p$ divides the index of $f(x)$, which contradicts (1). On the other hand, if $p\mid k$, then we can write $k=p^u\ell$, where $p\nmid \ell$. By Proposition \ref{C2}, $p$ divides the index of $f(x^{p^u})$, hence, by Corollary \ref{C1}, $p$ divides the index of $f(x^p)$. This contradicts (2), so $p$ must be a divisor of $f(0)$.
\end{proof}
\begin{proof}[Proof of Corollary \ref{abcde}]  
One can easily see that the monogenity of an irreducible polynomial of the type $f(x)=x^d+A\cdot h(x)\in\Z[x],$ with $d>1,$ and $\deg h(x)<d$, implies that $A$ is squarefree. Let $p\mid A$. Then $f(x)\in \langle p, x^2\rangle$, so by Corollary \ref{resqu} and Theorem 2.B, the monogenity of $f(x)$ implies that $A$ is squarefree. In this situation, $f(x)$ is an Eisenstein polynomial with respect to $p$ and hence by Lemma \ref{th2.1}, $p$ does not divide the index of $f(x)$. Therefore the proof is complete keeping Theorem \ref{sufnec} in mind.
\end{proof}
%\begin{proof}[Proof of Corollary \ref{TH3}] Note that $f(x^k)$ is irreducible, so  $f(x)$ is irreducible. Using Theorem \ref{sufnec}, the monogenity of $f(x^k)$ implies that both (1) and (2) holds. Conversely, suppose both $(1)$ and $(2)$ holds. If $A=\pm1$, then the result follows from Corollary \ref{TH2}. Let $p$ be a prime number dividing $f(0)=A$. If $p^2\mid A$ then Proposition \ref{re} shows that $p$ divides the index of $f(x^p)$, a contradiction to (2). Thus $A$ must be squarefree. Since $|h(0)|=1$ and $A$ is squarefree, we see that $f(x^k)$ is an Eisenstein polynomial with respect to $p$. Therefore by Lemma \ref{th2.1}, $p$ does not divide the index of $f(x^k)$. This shows that the index of $f(x^k)$ and $A$ are coprime. Hence Theorem \ref{TH2} implies that $f(x^k)$ is monogenic.\end{proof}
\begin{proof}[Proof of Theorem \ref{uvw}]  
    Set $\deg f(x)=d.$ Using Lemma \ref{dfpow}, we get 
    $$D(f(x^{p^u}))=\pm D(f(x))^{p^u}\cdot p^{udp^u}\cdot f(0)^{p^u-1}.$$
    (1) In view of Lemma \ref{abcd}, Equation \eqref{eq} and the hypothesis that $f(x)$ is monogenic, we conclude that any prime  divisor of the index of $f(x^{p^u})$ is either $p$ or a divisor of $f(0)$. If $|f(0)|=1$, then the index of $f(x^{p^u})$ is equal to $p^s$, for some non-negative integer $s$. Let $q\neq p$ be a prime divisor of $f(0)$. Since $f(x)$ is monogenic and $f(0)$ is squarefree,  Proposition \ref{C2} implies that $q$ does not divide the index of $f(x^{p^u})$. Hence, $p$ is the only possible prime divisor of the index of $f(x^{p^u})$.\\ 
(2) Using Corollary \ref{C1}, it is enough to prove that $p$ does not divide the index of $f(x^p)$ if and only if $\frac{f(x^p)-f(x)^p}{p}$ is coprime to $f(x)$ modulo $p$. Let $\overline{f}(x)=\displaystyle\prod_{i=1}^t\overline{g}^{e_i}_i(x)$ be the factorization of $\overline{f}(x)$ into a product of powers of distinct irreducible polynomials $\overline{g}_i(x)$, where each $g_i(x)\in\Z[x]$ monic. Here, $f(x^p)\in \langle p,g^2_i(x)\rangle$ for all $1\le i\le t$. We can write 
$$f(x)=\displaystyle\prod_{i=1}^tg_i^{e_i}(x)+p\cdot a(x),$$
for some $a(x)\in\Z[x]$. Then
$$f(x)^p=\bigg(\displaystyle\prod_{i=1}^tg_i^{e_i}(x)+p\cdot a(x)\bigg)^p=\bigg(\displaystyle\prod_{i=1}^tg_i^{e_i}(x)\bigg)^p+p^2\cdot t(x),$$ for some $t(x)\in\Z[x]$. Therefore, $f(x)^p\in \langle p^2, g_i(x) \rangle$ for each $1\le i \le t.$ Notice, that $f(x^p)\in \langle p^2,g_i(x)\rangle$ if and only if $f(x^p)-f(x)^p \in \langle p^2,g_i(x)\rangle.$  Set $f(x^p)-f(x)^p=p\cdot c(x),$ for some $c(x)\in\Z[x]$. Then, $f(x^p)\in\langle p^2, g_i(x)\rangle$ if and only if $p\cdot c(x)\in\langle p^2, g_i(x)\rangle$. Since  $g_i(x)\in\Z[x]$ is monic and irreducible modulo $p$ for each $i$, it is clear that  $p\cdot c(x)\in\langle p^2, g_i(x)\rangle$ if and only if $c(x)=g_i(x)\cdot u(x)+p\cdot v(x)$, for certain $u(x), v(x)\in\Z[x]$. Equivalently, $f(x^p)\in \langle p^2,g_i(x)\rangle$ if and only if $\frac{f(x^p)-f(x)^p}{p}$ is divisible by $g_i(x)$ modulo $p$. Thus, we get that $f(x^p)\not\in \langle p^2,g_i(x)\rangle$ for any $1\le i\le t$ if and only if $\frac{f(x^p)-f(x)^p}{p}$ is coprime to $f(x)$ modulo $p$. Using Corollary \ref{resqu}, $f(x^p)\not\in\langle p,g_i(x)\rangle^2$ for any $1\le i\le t$ if and only if $\frac{f(x^p)-f(x)^p}{p}$ is coprime to $f(x)$ modulo $p$. Hence, our result follows from Theorem 2.B.
\end{proof}

\begin{proof}[Proof of Theorem \ref{abc}] Let us start by proving that any prime divisor $p$ of $k$ does not divide the index of the irreducible polynomial $f(x^p)$ if and only if (2) holds.  Let $p$ be a prime divisor of $k$, and write
    \begin{align*}
			f(x^p)=&\big(x^d+A\cdot h(x)-A\cdot h(x)\big)^p+A\cdot h(x^p)\\
			=&\big(f(x)-A\cdot h(x)\big)^p+A\cdot h(x^p)\\
			% =&g^p(x)+(-1)^pA^p(Bx+1)^{mp}+\displaystyle\sum_{i=1}^{p-1}{p \choose i}g^i(x)(-A(Bx+1)^m)^{p-i}\\&+A(Bx^p+1)^m.\\
			=&f(x)^p+(-1)^pA^p\cdot h(x)^p+p\cdot f(x)\cdot r(x)+A\cdot h(x^p),
    \end{align*} 
    where 
    $$r(x)=\frac{1}{p}\bigg(\displaystyle\sum_{i=1}^{p-1}{p \choose i}f^{i-1}(x)\big(-A\cdot h(x)\big)^{p-i}\bigg).$$
    This implies, that 
    $$\frac{f(x^p)-f(x)^p}{p}=\frac{\big(-A\cdot h(x)\big)^p+p\cdot f(x)\cdot r(x)+A\cdot h(x^p)}{p}.$$ 
    Therefore, $\frac{f(x^p)-f(x)^p}{p}$ is coprime to $f(x)$ modulo $p$ if and only if
    $\frac{1}{p}[(-A\cdot h(x))^p+A\cdot h(x^p)]$ is coprime to $f(x)$ modulo $p$. Thus, using Theorem \ref{uvw}, $p$ does not divide the index of $f(x^p)$ if and only if $\frac{1}{p}[(-A\cdot h(x))^p+A\cdot h(x^p)]$ is coprime to $f(x)$ modulo $p.$ 
  
  Using Lemma \ref{L1}, the monogenity of $f(x^k)$ implies that both (1) and (2) hold. Conversely, suppose that both $(1)$ and $(2)$ hold. If $|A|=1$, then the result follows from Corollary \ref{TH2}. Let $p$ be a prime dividing $A$. If $p^2\mid A$ then Proposition \ref{re} shows that $p$ divides the index of $f(x^p)$, a contradiction to (2). Thus, $A$ must be squarefree. Since $|h(0)|=1$ and $A$ is squarefree, we can see that $f(x^k)$ is an Eisenstein polynomial with respect to $p$. Therefore, by Lemma \ref{th2.1}, $p$ does not divide the index of $f(x^k)$. This shows, that the index of $f(x^k)$ and $A$ are coprime. Hence, Corollary \ref{TH2} implies that $f(x^k)$ is monogenic.
  \end{proof}

\section{Infinite examples}
In this section, we introduce some well-known infinite parametric families of polynomials for which Theorem \ref{sufnec} can be applied efficiently. The first example is the family of pure polynomials obtained by choosing $f(x)=x-A$, where $A$ is an integer. In this case, our theorems imply the following well-known result (see \cite[Theorem 1.3]{jh-kh}, \cite[Corollary 1.5]{df}, \cite[Theorem 1.1]{gas}).

\begin{proposition}\label{abbb}
   Let $k>1$ be an integer and $f(x)=x-A\in\Z[x]$. If $f(x^k)$ is irreducible, then $f(x^k)$ is monogenic if and only if 
    \begin{enumerate}[(i)]
        \item $p^2\nmid A^p-A$ for any prime $p\mid k$ and
        \item $A$ is squarefree.
    \end{enumerate}
\end{proposition}

\begin{proof} 
    The polynomial $f(x)=x-A$ is always monogenic, so (1) of Theorem \ref{sufnec} is true for any $A\in\Z$. In view of Theorem \ref{sufnec}, it is enough to prove that (2) of Theorem \ref{sufnec} and (i) are equivalent. Here $f(x^p)\equiv (x-A)^p\mod p$. Theorem 2.B together with Corollary \ref{resqu} imply that $p$ divides the index of $f(x^p)$ if and only if the remainder of $f(x^p)$ divided by $x-A$, which is $A^p-A$, is a multiple of $p^2$. This completes the proof.
\end{proof}

One can easily generalize the result above. Assume that $p$ is a prime such that the polynomial $f(x)\in\Z[x]$ splits completely modulo $p$. Then $p$ divides the index of $f(x^{p^u})$ if and only if $p$ divides the index of $f(x^p)$. Thus, for any linear factor $(x-r)$ of $f(x)$ modulo $p$, we only have to check if $p^2$ divides $f(r^p)$ or not. This leads us to the following proposition.

\begin{proposition}\label{split}
	Let $f(x)\in\Z[x]$ and let $k$ be a positive integer, such that $f(x)$ splits completely modulo every prime divisor $p$ of $k$ and $f(x^k)$ is irreducible. Then $f(x^k)$ is monogenic if and only if
	\begin{enumerate}[(i)]
        \item $f(x)$ is monogenic, 
        \item $p^2\nmid f(r^p)$, for any prime $p\mid k$ and $r=0,\ldots,p-1$, and
		\item $f(0)$ is squarefree.
    \end{enumerate}
\end{proposition}

We can apply the proposition above to infinite parametric families of number fields. We present this method on the family of simplest cubic polynomials of the form
$$f(x)=x^3-mx^2-(m+3)x-1.$$ 

\begin{proposition}\label{hi}
    Let $m\in\Z$ and $f(x)=x^3-mx^2-(m+3)x-1$. Let $k$ be a positive integer such that $f(x)$ is reducible modulo every odd prime $p\mid k$. Then $f(x^k)$ is monogenic if and only if
    \begin{enumerate}[(i)]
        \item $f(x)$ is monogenic, and
        \item $m\not\equiv\dfrac{r^{3p}-3r^p-1}{r^p(r^p+1)}\mod p^2$, for any odd prime $p\mid k$ and $r=1,\ldots,p-2$. 
    \end{enumerate}
\end{proposition}
\begin{proof}
    We note, that $f(x^k)$ is always irreducible in $\Z[x]$ (see  \cite[Lemma 3.1]{ljj}).
    Referring back to Theorem  \ref{sufnec}, we only have to prove that (2) of Theorem \ref{sufnec} is equivalent to  (ii). Let $p$ be an odd prime. The Galois group of $f(x)$ is the cyclic group of order three, so the reducibility of $f(x)$ modulo $p$ implies that it splits into three linear factors. By Proposition \ref{split}, we have to show that $p^2\nmid f(r^p)$, for any prime $p\mid k$ and $r=0,\ldots,p-1$. 

	Since $f(0)=-1$ and $f(-1)=1$, then $p^2\nmid f(r^p)$ if $r\equiv 0,-1\mod p$. For the other values $r=1,\ldots,p-2$, we have $f(r^p)=r^{3p}-mr^{2p}-(m+3)r^p-1\equiv 0\mod {p^2}$ if and only if
    $$m\equiv\dfrac{r^{3p}-3r^p-1}{r^p(r^p+1)}\mod{p^2}.$$
	Hence, $p$ does not divide the index of $f(x^p)$ if and only if (ii) holds.
    Finally, it can be easily checked that for $p=2$, $f(x)$ is always irreducible modulo 2, and $f(x^2)$ is not contained in $\langle2,f(x)\rangle^2$, so $2$ can 
    not divide the index of $f(x^k)$.
\end{proof}

\begin{remark}
    In Proposition \ref{hi}, the condition $f(x)$ is reducible modulo every odd prime $p\mid k$ can not be dropped. For example, if $k=13$ and $m\equiv 6\mod{13}$, then $f(x)$ is irreducible modulo $13$, however, if $m=71$, then both (i) and (ii) are true, but 13 divides the index of $f(x^k)$. 
    For the sake completeness, see \cite[Lemma 3.3]{ljj}, in which the author gives a characterization of the monogenity of $f(x^p)$, where $f(x)=x^3-mx^2-(m+3)x-1$ is irreducible modulo $p$.
\end{remark}

In the next example, we apply Theorem \ref{abc} to $f(x)=x^d+A(Bx+1)^m$. The monogenity of $f(x)$ is investigated in \cite{susu}.

\begin{proposition}
	Let $A$ and $B$ be integers and $d$ and $m$ positive integers such that $d>m$. Assume, that  $\gcd(d,mB)=1.$ Let $f(x)=x^d+A(Bx+1)^m$ and let $k$ be a positive integer such that $f(x^k)$ is irreducible. Then $f(x^k)$ is monogenic if and only if the following statements hold:
	\begin{enumerate}
		\item Both $A$ and $d^d+(-1)^{d+m}B^dm^m(d-m)^{d-m}A$ are squarefree.
		\item For any prime $p$ dividing $k$, 
            $$\frac{1}{p}[\big(-A(Bx+1)^m\big)^p+A(Bx^p+1)^m]$$
            is coprime to $f(x)$ modulo $p.$
	\end{enumerate}
 Furthermore, if $\rad(k) \mid \rad(A)$, then $f(x^k)$ is monogenic if and only if both $A$ and $d^d+(-1)^{d+m}B^dm^m(d-m)^{d-m}A$ are squarefree.
\end{proposition}
\begin{proof}
If $A\neq \pm1$ is a squarefree integer, then $f(x^k)$ is an Eisenstein polynomial with respect to a prime divisor of $k$, thus $f(x^k)$ is irreducible. This implies, that there exist an infinite number of irreducible polynomials of the form $f(x^k)=x^{dk}+A(Bx^k+1)^m$. By \cite[Theorem 1.1]{susu}, $f(x)$ is monogenic if and only if both $A$ and $d^d+(-1)^{d+m}C^dm^m(d-m)^{d-m}A$ are squarefree. Set $h(x)=(Bx+1)^m$, then $h(0)=1$ and $\deg h(x)<d$, and our result follows from Theorem \ref{abc}. \end{proof}
 
 \begin{comment}
 Next, let $p$ be a prime divisor of $k$.	Write\begin{align*}
			f(x^p)=&(x^d+A(Bx+1)^m-A(Bx+1)^m)^p+A(Bx^p+1)^m\\
			=&(f(x)-A(Bx+1)^m)^p+A(Bx^p+1)^m\\
			% =&g^p(x)+(-1)^pA^p(Bx+1)^{mp}+\displaystyle\sum_{i=1}^{p-1}{p \choose i}g^i(x)(-A(Bx+1)^m)^{p-i}\\&+A(Bx^p+1)^m.\\
			=&f(x)^p+(-1)^pA^p(Bx+1)^{mp}+pf(x)r(x)+A(Bx^p+1)^m,
		\end{align*} where $r(x)=\frac{1}{p}\bigg(\displaystyle\sum_{i=1}^{p-1}{p \choose i}f^{i-1}(x)(-A(Bx+1)^m)^{p-i}\bigg).$ This implies that $$\frac{f(x^p)-f(x)^p}{p}=\frac{1}{p}[(-A(Bx+1)^m)^p+pf(x)r(x)+A(Bx^p+1)^m].$$ Clearly $\frac{f(x^p)-f(x)^p}{p}$ is coprime to $f(x)$ modulo $p$ if and only if the polynomial $\frac{1}{p}[(-A(Bx+1)^m)^p+A(Bx^p+1)^m]$ is coprime to $f(x)$ modulo $p.$ Thus using Theorem \ref{uvw}, $p$ does not divide the index of $f(x^p)$ if and only if $\frac{1}{p}[(-A(Bx+1)^m)^p+A(Bx^p+1)^m]$ is coprime to $f(x)$ modulo $p.$ Hence the result follows from Theorem \ref{sufnec}. 
  
  If $\rad(k)$ divides $\rad(A)$, then $f(x)\equiv x^d \mod p$ for all $p$ divding $k$.  By (1), we can conclude that $A$ is square free. So the polynomial $\frac{1}{p}[(-A(Bx+1)^m)^p+A(Bx^p+1)^m]$ is always coprime to $f(x)$ modulo $p$.\end{comment}

In the last example, there is no parameter in the polynomial $f(x)$, but it shows how we can create infinite monogenic examples using Theorem \ref{sufnec}.

\begin{proposition}
    Let $f(x)=x^2-x-1$ and $k$ be a positive integer, such that all of the prime factors of $k$ are less than 10000 and  $f(x^k)$ is irreducible. Then $f(x^k)$ is monogenic
\end{proposition}

\begin{proof}
    Clearly $f(x)$ is monogenic and $f(0)=-1$ is squarefree, so conditions (1) and (3) of Theorem \ref{sufnec} hold. For any prime $p$ less than 10000 we checked by MAPLE that $p$ does not divide the index of $f(x^p)$. So by Theorem \ref{sufnec}, if all prime factors of $k$ are less than 10000, then $f(x^k)$ is monogenic.
\end{proof}

\begin{remark}
    By the results of L. Jones \cite{ljsun}, $p$ divides the index of $f(x^p)$ if and only if $p$ is a Wall-Sun-Sun prime. Since there is no such prime under $2^{64}$, it is possible to change 10000 in the above theorem to $2^{64}$.
\end{remark}

\end{document}